\documentclass[english,12pt]{article}
\usepackage[T1]{fontenc}
\usepackage[latin9]{inputenc}
\usepackage{color}
\usepackage{babel}
\usepackage{units}
\usepackage{bm}
\usepackage{amsmath}
\usepackage{amsthm}
\usepackage{amssymb}
\usepackage[unicode=true,
 bookmarks=true,bookmarksnumbered=false,bookmarksopen=false,
 breaklinks=false,pdfborder={0 0 1},backref=false,colorlinks=false]
 {hyperref}
\hypersetup{pdftitle={Local times for ggBm in higher dimensions},
 pdfauthor={Jos{\'e} Lu{\'\i}s da Silva},
 pdfsubject={Local times of gBm},
 pdfkeywords={Local times, grey Brownain motion},
 linkcolor=red, citecolor=red, urlcolor=blue, filecolor=blue}
\allowdisplaybreaks
\makeatletter
\theoremstyle{plain}
\newtheorem{thm}{\protect\theoremname}
  \theoremstyle{plain}
  \newtheorem{prop}[thm]{\protect\propositionname}
  \theoremstyle{definition}
  \newtheorem{defn}[thm]{\protect\definitionname}
  \theoremstyle{plain}
  \newtheorem{rem}[thm]{\protect\remarkname}
  \theoremstyle{plain}
  \newtheorem{lem}[thm]{\protect\lemmaname}
  \theoremstyle{plain}

\usepackage{color,pdfsync,graphics,lineno,hyperref}
\makeatother

  \providecommand{\conjecturename}{Conjecture}
  \providecommand{\definitionname}{Definition}
  \providecommand{\lemmaname}{Lemma}
  \providecommand{\propositionname}{Proposition}
  \providecommand{\remarkname}{Remark}
\providecommand{\theoremname}{Theorem}

\begin{document}

\title{Self-intersection local times for generalized grey Brownian motion in higher dimensions}

\author{\textbf{Jos{\'e} Lu{\'\i}s da Silva}\\
CIMA, University of Madeira, Campus da Penteada,\\
9020-105 Funchal, Portugal\\
Email: luis@uma.pt
\and
\textbf{Herry Pribawanto Suryawan}\\
Department of Mathematics, Sanata Dharma University\\
55281 Yogyakarta, Indonesia\\
and\\
Institut f{\"u}r Mathematik, Universit{\"a}t Z{\"u}rich\\
CH-8057, Z{\"u}rich, Switzerland\\
Email: herrypribs@usd.ac.id, herry.suryawan@math.uzh.ch
\and
\textbf{Wolfgang Bock}\\
Technische Universit{\"a}t Kaiserslautern,\\
Fachbereich Mathematik, Postfach 3049,\\
67653 Kaiserslautern, Germany\\
Email: bock@mathematik.uni-kl.de
}


\maketitle

\begin{abstract}
We prove that the self-intersection local times for generalized grey Brownian motion $B^{\beta,\alpha}$ in arbitrary dimension $d$ is a well defined object in a suitable distribution space for $d\alpha<2$.
\end{abstract}

\section{Introduction}
Intersection local times are an intensively studied object for about 80 years, see e.g.~\cite{levy}. Heuristically the self-intersection measures the time the process spends on its trajectory, i.e., it serves to count the self-crossings of the trajectory of a random process. In an informal way the self-intersection local time can be expressed by
\[
L(Y)\equiv \int d^2t\,\delta (Y(t_2)-Y(t_1)), 
\]
where $\delta$ is Donsker's - $\delta $-function and $Y$ a random process. Indeed the random variable $L$ is intended to sum up the contributions from each pair of ''times'' $t_1,t_2$
for which the process $Y$ is at the same point.
For Gaussian processes the self-intersection local time is defined as a rigorous object, see for the case of Brownian motion e.g.~\cite{ARHZ, bass, fcs, dvor2, he, imke, legall, lyons, sym, varadhan, watanabe, west, wolp, yor2} and e.g.~\cite{Rosen87, HN05, ayache2011, drumond-oliveira-silva08, GOSS10, meerschaert2008, Bock:2014wn} for fractional and multifractional Brownian motion. 
One framework which serves to give a mathematical sound meaning to the object above in the Gaussian setting is White Noise Analysis, see e.g.~\cite{HKPS93, Ob94, Kuo96}.   

For non-Gaussian processes in \cite{GJRS14}  a similar concept was used to establish the Mittag-Leffler Analysis.  The grey noise measure \cite{Sch92, MM09} is included as a special case in the class of Mittag-Leffler measures, which offers the possibility to apply the Mittag-Leffler analysis to fractional differential equations, in particular to fractional diffusion equations \cite{Schneider90, Sch92}, which carry numerous applications in science, like relaxation type differential equations or viscoelasticity. 

The corresponding grey Brownian motion (gBm) was introduced by W. Schneider as a model for slow anomalous diffusions, i.e., the marginal density function of the gBm is the fundamental solution of a time-fractional diffusion equation. This is a class $\left\{ B_{\beta}(t): t\ge 0, 0<\beta\le 1\right\}$ of stochastic processes which are self-similar with stationary increments. More recently, this class was extended to the, so called generalized grey Brownian motion (ggBm) to include slow and fast anomalous diffusions which contain either Gaussian or non-Gaussian processes e.g., grey Brownian motion and fractional Brownian motion. In this paper we study the existence of self-intersection local times of ggBm in dimension $d$.

In Section 2 we summarize the construction and basic properties of ggBm in dimension $d$. Section 3 contains the main result on the existence of self-intersection local times of a $d$-dimensional ggBm as a weak integral in a suitable stochastic distribution space.

\section{The Mittag-Leffler Measure}

\label{sec:ggBm_high_d}

Let $d\in\mathbb{N}$ and $L_{d}^{2}$
be the Hilbert space of vector-valued square integrable measurable
functions
\[
L_{d}^{2}:=L^{2}(\mathbb{R})\otimes\mathbb{R}^{d}.
\]
The space $L_{d}^{2}$ is unitary isomorphic to a direct sum of $d$
identical copies of $L^{2}:=L^{2}(\mathbb{R})$, (i.e., the space
of real-valued square integrable measurable functions with Lebesgue
measure). Any element $f\in L_{d}^{2}$ may be written in the form
\begin{equation}
f=(f_{1}\otimes e_{1},\ldots,f_{d}\otimes e_{d}),\label{eq:L2d_element}
\end{equation}
where $f_{i}\in L^{2}(\mathbb{R})$, $i=1,\ldots,d$ and $\{e_{1},\ldots,e_{d}\}$
denotes the canonical basis of $\mathbb{R}^{d}$. The inner product
in $L_{d}^{2}$ is given by
\[
(f,g)_{0}=\sum_{k=1}^{d}(f_{k},g_{k})_{L^{2}}=\sum_{k=1}^{d}\int_{\mathbb{R}}f_{k}(x)g_{k}(x)\,dx,
\]
where $g=(g_{1}\otimes e_{1},\ldots,g_{d}\otimes e_{d})$, $f_{k}\in L^{2}$,
$k=1,\ldots,d$, $f$ as given in (\ref{eq:L2d_element}). The corresponding
norm in $L_{d}^{2}$ is given by
\[
|f|_{0}^{2}:=\sum_{k=1}^{d}|f_{k}|_{L^{2}}^{2}=\sum_{k=1}^{d}\int_{\mathbb{R}}f_{k}^{2}(x)\,dx.
\]

As a densely embedded nuclear Fr\'echet space in $L_{d}^{2}$ we choose $S_{d}:=S(\mathbb{R})\otimes\mathbb{R}^{d}$,
where $S(\mathbb{R})$ is the Schwartz test function space. An element
$\varphi\in S_{d}$ has the form
\begin{equation}
\varphi=(\varphi_{1}\otimes e_{1},\ldots,\varphi_{d}\otimes e_{d}),\label{eq:test_function.}
\end{equation}
where $\varphi_{i}\in S(\mathbb{R})$, $i=1,\ldots,d$. Together with
the dual space $S'_{d}:=S'(\mathbb{R})\otimes\mathbb{R}^{d}$ we obtain
the basic Gel'fand triple
\[
S_{d}\subset L_{d}^{2}\subset S'_{d}.
\]
The dual pairing between $S'_{d}$ and $S_{d}$ is given as an extension
of the scalar product in $L_{d}^{2}$ by 
\[
\langle f,\varphi\rangle_{0}=\sum_{k=1}^{d}(f_{k},\varphi_{k})_{L^{2}},
\]
where $f$ and $\varphi$ as in (\ref{eq:L2d_element}) and (\ref{eq:test_function.}),
respectively. In $S'_{d}$ we choose the Borel $\sigma$-algebra $\mathcal{B}$
generated by the cylinder sets. 

Define the operator $M_{-}^{\nicefrac{\alpha}{2}}$ on $S(\mathbb{R})$
by
\[
M_{-}^{\nicefrac{\alpha}{2}}\varphi:=\begin{cases}
K_{\nicefrac{\alpha}{2}}D_{-}^{-(\alpha-1)/2}\varphi, & \alpha\in(0,1),\\
\varphi, & \alpha=1,\\
K_{\nicefrac{\alpha}{2}}I_{-}^{(\alpha-1)/2}\varphi, & \alpha\in(1,2),
\end{cases}
\]
where the normalization constant $K_{\nicefrac{\alpha}{2}}:=\sqrt{\alpha\sin(\nicefrac{\alpha\pi}{2})\Gamma(\alpha)}$
and $D_{-}^{r}$, $I_{-}^{r}$ denote the left-side fractional derivative
and fractional integral of order $r$ in the sense of Riemann-Liouville,
respectively:
\begin{eqnarray*}
(D_{-}^{r}f)(x)&=&\frac{-1}{\Gamma(1-r)}\frac{d}{dx}\int_x^{\infty}f(t)(t-x)^{-r}\, dt\\
(I_{-}^{r}f)(x)&=&\frac{1}{\Gamma(r)}\int_x^{\infty}f(t)(t-x)^{r-1}\, dt, \quad x\in \mathbb{R}.
\end{eqnarray*}
We refer to \cite{SKM1993} or \cite{KST2006} for the
details on these operators. It is possible to obtain a larger domain
of the operator $M_{-}^{\nicefrac{\alpha}{2}}$ in order to include
the indicator function $1\!\!1_{[0,t)}$ such that $M_{-}^{\nicefrac{\alpha}{2}}1\!\!1_{[0,t)}\in L^{2}$,
for the details we refer to Appendix A in \cite{GJ15}. We have the
following 
\begin{prop}[Corollary 3.5 in \cite{GJ15}]
For all $t,s\ge0$ and all $0<\alpha<2$ it holds that
\begin{equation}
\big(M_{-}^{\nicefrac{\alpha}{2}}1\!\!1_{[0,t)},M_{-}^{\nicefrac{\alpha}{2}}1\!\!1_{[0,s)}\big)_{L^{2}}=\frac{1}{2}\big(t^{\alpha}+s^{\alpha}-|t-s|^{\alpha}\big).\label{eq:alpha-inner-prod}
\end{equation}
\end{prop}
Note that this coincides with the covariance of the fractional Brownian motion with Hurst parameter $H=\frac{\alpha}{2}$.

In order to construct ggBm we will use the Mittag-Leffler function which is introduced by G.\ Mittag-Leffler
in a series of papers \cite{Mittag-Leffler1903,Mittag-Leffler1904,Mittag-Leffler1905}. 
\begin{defn}[Mittag-Leffler function]
\label{def:MLf} For $\beta>0$ the Mittag-Leffler function $E_{\beta}$ 
is defined as an entire function by the following series representation
\begin{equation}
E_{\beta}(z):=\sum_{n=0}^{\infty}\frac{z^{n}}{\Gamma(\beta n+1)},\quad z\in\mathbb{C},\label{eq:MLf}
\end{equation}
where $\Gamma$ denotes the Euler gamma function.
\end{defn}
\noindent Note that for $\beta=1$ the Mittag-Leffler function coincides with the classical exponential function. We also consider the so-called the $M$-Wright function $M_{\beta}$, $0<\beta\le1$ (in one variable)
defined by 
\[
M_{\beta}(z):=\sum_{n=0}^{\infty}\frac{(-z)^{n}}{n!\Gamma(-\beta n+1-\beta)}.
\]
For the choice $\beta=\nicefrac{1}{2}$ the corresponding $M$-Wright function reduces to the Gaussian density
\begin{equation}
M_{\nicefrac{1}{2}}(z)=\frac{1}{\sqrt{\pi}}\exp\left(-\frac{z^{2}}{4}\right).\label{eq:MWright_Gaussian}
\end{equation}
The Mittag-Leffler function $E_{\beta}$ and the $M$-Wright are related through the Laplace transform
\begin{equation}
\int_{0}^{\infty}e^{-s\tau}M_{\beta}(\tau)\,d\tau=E_{\beta}(-s).\label{eq:LaplaceT_MWf}
\end{equation}

The Mittag-Leffler measures $\mu_{\beta}$, $0<\beta\leq1$
is a family of probability measures on $S'_{d}$ whose characteristic
functions are given by Mittag-Leffler functions, see Definition\ \ref{def:MLf}.
Using the Bochner-Minlos theorem, see \cite{Gelfand-Vilenkin64} or \cite{H70},
we obtain the following definition.
\begin{defn}[cf.\ \cite{GJRS14}]
For any $\beta\in(0,1]$ the Mittag-Leffler measure is defined as
the unique probability measure $\mu_{\beta}$ on $S'_{d}$ by fixing
its characteristic functional
\begin{equation}
\int_{S'_{d}}e^{i\langle w,\varphi\rangle_{0}}\,d\mu_{\beta}(w)=E_{\beta}\left(-\frac{1}{2}|\varphi|_{0}^{2}\right),\quad\varphi\in S_{d}.\label{eq:ch-fc-gnm}
\end{equation}
\end{defn}

\begin{rem}
\label{rem:grey-noise-measure}
\begin{enumerate}
\item The measure $\mu_{\beta}$ is also called grey noise (reference) measure,
cf.\ \cite{GJRS14} and \cite{GJ15}.
\item The range $0<\beta\leq1$ ensures the complete monotonicity of $E_{\beta}(-x)$,
see Pollard \cite{Pollard48}, i.e., $(-1)^{n}E_{\beta}^{(n)}(-x)\ge0$
for all $x\ge0$ and $n\in\mathbb{N}_{0}:=\{0,1,2,\ldots\}.$ In other
words, it is sufficient to show that 
\[
S_{d}\ni\varphi\mapsto E_{\beta}\left(-\frac{1}{2}|\varphi|_{0}^{2}\right)\in\mathbb{R}
\]
is a characteristic function in $S_{d}$. 
\end{enumerate}
\end{rem}

By $L^{2}(\mu_{\beta}):=L^{2}(S'_{d},\mathcal{B},\mu_{\beta})$ we denote the complex Hilbert space of square integrable measurable
functions defined on $S'_{d}$ with scalar product
\[
(\!(F,G)\!)_{L^{2}(\mu_{\beta})}:=\int_{S'_{d}}F(w)\bar{G}(w)\,d\mu_{\beta}(w),\quad F,G\in L^{2}(\mu_{\beta}).
\]
The corresponding norm is denoted by $\|\cdot\|_{L^{2}(\mu_{\beta})}$.
It follows from (\ref{eq:ch-fc-gnm}) that all moments of $\mu_{\beta}$
exists and we have
\begin{lem}
\label{lem:gnm}For any $\varphi\in S_{d}$ and $n\in\mathbb{N}_{0}$
we have
\begin{align*}
\int_{S'_{d}}\langle w,\varphi\rangle_{0}^{2n+1}\,d\mu_{\beta}(w) & =0,\\
\int_{S'_{d}}\langle w,\varphi\rangle_{0}^{2n}\,d\mu_{\beta}(w) & =\frac{(2n)!}{2^{n}\Gamma(\beta n+1)}|\varphi|_{0}^{2n}.
\end{align*}
In particular, $\|\langle\cdot,\varphi\rangle\|_{L^{2}(\mu_{\beta})}^{2}=\frac{1}{\Gamma(\beta+1)}|\varphi|_{0}^{2}$
and by polarization for any $\varphi,\psi\in S_{d}$ we obtain
\[
\int_{S'_{d}}\langle w,\varphi\rangle_{0}\langle w,\psi\rangle_{0}\,d\mu_{\beta}(w)=\frac{1}{\Gamma(\beta+1)}\langle\varphi,\psi\rangle_{0}.
\]
\end{lem}

\section{Generalized grey Brownian motion in dimension $d$}

\label{subsec:ggBm}For any test function $\varphi\in S_{d}$ we define
the random variable
\[
X^{\beta}(\varphi):S'_{d}\longrightarrow\mathbb{R}^{d},\;w\mapsto X^{\beta}(\varphi,w):=\big(\langle w_{1},\varphi_{1}\rangle,\ldots,\langle w_{d},\varphi_{d}\rangle\big).
\]
The random variable $X^{\beta}(\varphi)$ has the following properties
which are a consequence of Lemma\ \ref{lem:gnm} and the characteristic
function of $\mu_{\beta}$ given in (\ref{eq:ch-fc-gnm}).
\begin{prop}
Let $\varphi,\psi\in S_{d}$, $k\in\mathbb{R}^{d}$ be given. Then
\begin{enumerate}
\item The characteristic function of $X^{\beta}(\varphi)$ is given by 
\begin{equation}
\mathbb{E}\big(e^{i(k,X^{\beta}(\varphi))}\big)=E_{\beta}\left(-\frac{1}{2}\sum_{j=1}^{d}k_{j}^{2}|\varphi_{j}|_{L^{2}}^{2}\right).\label{eq:characteristic-coord-proc}
\end{equation}
\item The characteristic function of the random variable $X^{\beta}(\varphi)-X^{\beta}(\psi)$
is
\begin{equation}
\mathbb{E}\big(e^{i(k,X^{\beta}(\varphi)-X^{\beta}(\psi))}\big)=E_{\beta}\left(-\frac{1}{2}\sum_{i=1}^{d}k_{j}^{2}|\varphi_{j}-\psi_{j}|_{L^{2}}^{2}\right).\label{eq:CF_increment}
\end{equation}
\item The expectation of the $X^{\beta}(\varphi)$ is zero and 
\begin{equation}
\|X^{\beta}(\varphi)\|_{L^{2}(\mu_{\beta})}^{2}=\frac{1}{\Gamma(\beta+1)}|\varphi|_{0}^{2}.\label{eq:variance.cood-proc}
\end{equation}
\item The moments of $X^{\beta}(\varphi)$ are given by
\begin{align*}
\int_{S'_{d}}\big|X^{\beta}(\varphi,w)\big|^{2n+1}\,d\mu_{\beta}(w) & =0,\\
\int_{S'_{d}}\big|X^{\beta}(\varphi,w)\big|^{2n}\,d\mu_{\beta}(w) & =\frac{(2n)!}{2^{n}\Gamma(\beta n+1)}|\varphi|_{0}^{2n}.
\end{align*}
\end{enumerate}
\end{prop}

\begin{rem}\label{rem:l2_ggbm_ext}
\begin{enumerate}
\item The property (\ref{eq:variance.cood-proc}) of $X^{\beta}(\varphi)$
gives the possibility to extend the definition of $X^{\beta}$ to
any element in $L_{d}^{2}$, in fact, if $f\in L_{d}^{2}$, then there
exists a sequence $(\varphi_{k})_{k=1}^{\infty}\subset S_{d}$ such
that $\varphi_{k}\longrightarrow f$, $k\rightarrow\infty$ in the
norm of $L_{d}^{2}$. Hence, the sequence $\big(X^{\beta}(\varphi_{k})\big)_{k=1}^{\infty}\subset L^{2}(\mu_{\beta})$
forms a Cauchy sequence which converges to an element denoted by $X^{\beta}(f)\in L^{2}(\mu_{\beta})$. 
\item For $\beta=1$ property (\ref{eq:variance.cood-proc}) yields the It\^o isometry.
\end{enumerate}
\end{rem}

\noindent We define $1\!\!1_{[0,t)}\in L_{d}^{2}$, $t\ge0$, by
\[
1\!\!1_{[0,t)}:=(1\!\!1_{[0,t)}\otimes e_{1},\ldots,1\!\!1_{[0,t)}\otimes e_{d})
\]
 and consider the process $X^{\beta}(1\!\!1_{[0,t)})\in L^{2}(\mu_{\beta})$
such that the following definition makes sense.
\begin{defn}
For any $0<\alpha<2$ we define the process 
\begin{align}
S'_{d}\ni w\mapsto B^{\beta,\alpha}(t,w) & :=\big(\langle w,(M_{-}^{\nicefrac{\alpha}{2}}1\!\!1_{[0,t)})\otimes e_{1}\rangle,\ldots,\langle w,(M_{-}^{\nicefrac{\alpha}{2}}1\!\!1_{[0,t)})\otimes e_{d}\rangle\big)\nonumber \\
 & =\big(\langle w_{1},M_{-}^{\nicefrac{\alpha}{2}}1\!\!1_{[0,t)}\rangle,\ldots,\langle w_{d},M_{-}^{\nicefrac{\alpha}{2}}1\!\!1_{[0,t)}\rangle\big),\;t>0\label{eq:ggBm}
\end{align}
as an element in $L^{2}(\mu_{\beta})$. This process is called a version of $d$-dimensional
generalized grey Brownian motion (ggBm). Its characteristic function
has the form
\begin{equation}
\mathbb{E}\big(e^{i(k,B^{\beta,\alpha}(t))}\big)=E_{\beta}\left(-\frac{|k|^{2}}{2}t^{\alpha}\right),\;k\in\mathbb{R}^{d}.\label{eq:ch-fc-ggBm}
\end{equation}
\end{defn}

\begin{rem}
\begin{enumerate}
\item By Remark \ref{rem:l2_ggbm_ext} the $d$-dimensional ggBm exist as a $L^2(\mu_{\beta})$-limit and hence the map $S'_d\ni \omega \mapsto \langle \omega,1\!\!1_{[0,t)} \rangle$ yields a version of ggBm, $\mu_{\beta}-a.s.$, but not in the pathwise sense. 
\item For a fixed $0<\alpha<2$ one can show by the Kolmogorov-Centsov continuity theorem that the paths of the process are $\mu_{\beta}-a.s.$ continuous. 
\end{enumerate}
\end{rem}

\begin{prop}
For any $0<\alpha<2$, the process $B^{\beta,\alpha}:=\{B^{\beta,\alpha}(t),\;t\ge0\}$,
is $\nicefrac{\alpha}{2}$ self-similar with stationary increments.
\end{prop}

\begin{proof}
Given $k=(k_{1},k_{2},\ldots,k_{n})\in\mathbb{R}^{n}$, we have to
show that for any $0<t_{1}<t_{2}<\ldots<t_{n}$ and $a>0$:
\begin{equation}
\mathbb{E}\left(\exp\left(i\left\langle \cdot,\sum_{j=1}^{n}k_{j}M_{-}^{\nicefrac{\alpha}{2}}1\!\!1_{[0,at_{j})}\right\rangle \right)\right)=\mathbb{E}\left(\exp\left(ia^{\nicefrac{\alpha}{2}}\left\langle \cdot,\sum_{j=1}^{n}k_{j}M_{-}^{\nicefrac{\alpha}{2}}1\!\!1_{[0,t_{j})}\right\rangle \right)\right).\label{eq:equality-ch-fc}
\end{equation}
It follows from (\ref{eq:characteristic-coord-proc}) that eq.\ (\ref{eq:equality-ch-fc})
is equivalent to
\[
E_{\beta}\left(-\frac{1}{2}\left|\sum_{j=1}^{n}k_{j}M_{-}^{\nicefrac{\alpha}{2}}1\!\!1_{[0,at_{j})}\right|_{L^{2}}^{2}\right)=E_{\beta}\left(-\frac{1}{2}\left|a^{\nicefrac{\alpha}{2}}\sum_{j=1}^{n}k_{j}M_{-}^{\nicefrac{\alpha}{2}}1\!\!1_{[0,t_{j})}\right|_{L^{2}}^{2}\right).
\]
Because of the complete monotonicity of $E_{\beta}$, the above equality
reduces to 
\[
\left|\sum_{j=1}^{n}k_{j}M_{-}^{\nicefrac{\alpha}{2}}1\!\!1_{[0,at_{j})}\right|_{L^{2}}^{2}=a^{\alpha}\left|\sum_{j=1}^{n}k_{j}M_{-}^{\nicefrac{\alpha}{2}}1\!\!1_{[0,t_{j})}\right|_{L^{2}}^{2}
\]
which is easy to show, taking into account (\ref{eq:alpha-inner-prod}).
A similar procedure may be applied in order to prove the stationarity
of the increments. Hence, for any $h\ge0$, we have to show that
\[
\mathbb{E}\left(\exp\left(i\sum_{j=1}^{n}k_{j}\big(B^{\beta,\alpha}(t_{j}+h)-B^{\beta,\alpha}(h)\big)\right)\right)=\mathbb{E}\left(\exp\left(i\sum_{j=1}^{n}k_{j}B^{\beta,\alpha}(t_{j})\right)\right).
\]
The above procedure reduces this equality to check the following 
\[
\left|\sum_{j=1}^{n}k_{j}M_{-}^{\nicefrac{\alpha}{2}}1\!\!1_{[h,t_{j}+h)}\right|_{L^{2}}^{2}=\left|\sum_{j=1}^{n}k_{j}M_{-}^{\nicefrac{\alpha}{2}}1\!\!1_{[0,t_{j})}\right|_{L^{2}}^{2}
\]
which is verified analogously as for the self-similarity.
\end{proof}
\begin{rem}
\label{rem:self-similar}The family $\{B^{\beta,\alpha}(t),\;t\ge0,\,\beta\in(0,1],\,\alpha\in(0,2)\}$
forms a class of $\nicefrac{\alpha}{2}$ self-similar process with
stationary increments ($\nicefrac{\alpha}{2}$-sssi) which includes:
\begin{enumerate}
\item For $\beta=\alpha=1$, the process $\{B^{1,1}(t),\;t\ge0\}$ is a
standard $d$-dimensional Brownian motion.
\item For $\beta=1$ and $0<\alpha<2$, $\{B^{1,\alpha}(t),\;t\ge0\}$ is
a $d$-dimensional fractional Brownian motion with Hurst parameter $\nicefrac{\alpha}{2}$.
\item For $\alpha=1$, $\{B^{\beta,1}(t),\;t\ge0\}$ is $\nicefrac{1}{2}$
self-similar non Gaussian process with
\begin{equation}
\mathbb{E}\left(e^{i\left(k,B^{\beta,1}(t)\right)}\right)=E_{\beta}\left(-\frac{|k|^{2}}{2}t\right),\quad k\in\mathbb{R}^{d}.\label{eq:ch-fc-1/2sssi}
\end{equation}
\item For $0<\alpha=\beta<1$, the process $\{B^{\beta}(t):=B^{\beta,\beta}(t),\;t\ge0\}$
is $\nicefrac{\beta}{2}$ self-similar and is called $d$-dimensional
grey Brownian motion (gBm for short). Its characteristic function
is given by 
\begin{equation}
\mathbb{E}\left(e^{i\left(k,B^{\beta}(t)\right)}\right)=E_{\beta}\left(-\frac{|k|^{2}}{2}t^{\beta}\right),\quad k\in\mathbb{R}^{d}.\label{eq:ch-fc-gBm}
\end{equation}
For $d=1$, gBm was introduced by W.\ Schneider in \cite{Schneider90,Sch92}.
\end{enumerate}
\end{rem}

\section{Distributions and characterization theorems}

\noindent There is a standard way to construct the test and distribution
spaces in non Gaussian analysis through Appell systems, the details
of this construction can be found in \cite{KSWY95}, \cite{GJRS14}, \cite{GJ15}
and references therein. In between the many choices of triples which
can be constructed we choose the Kondratiev triple 
\[
(S_d)_{\mu_{\beta}}^{1}\subset(H_{p})_{q,\mu_{\beta}}^{1}\subset L^{2}(\mu_{\beta})\subset(H_{-p})_{-q,\mu_{\beta}}^{-1}\subset(S_d)_{\mu_{\beta}}^{-1}.
\]
The space $(H_{p})_{q,\mu_{\beta}}^{1}$ is defined as the completion
of the $\mathcal{P}\big(S_d'\big)$ (the space of smooth
polynomials on $S_d'$) w.r.t.\ the norm $\|\cdot\|_{p,q,\mu_{\beta}}$
given by 
\[
\|\varphi\|_{p,q,\mu_{\beta}}^{2}:=\sum_{n=0}^{\infty}(n!)^{2}2^{nq}|\varphi^{(n)}|_{p}^{2},\quad p,q\in\mathbb{N}_{0},\;\varphi\in\mathcal{P}\big(S_d'\big).
\]
The dual space $(H_{-p})_{-q,\mu_{\beta}}^{-1}$ is a subset of $\mathcal{P}'\big(S_d'\big)$
such that if $\Phi\in(H_{-p})_{-q,\mu_{\beta}}^{-1}$, then 
\[
\|\Phi\|_{-p,-q,\mu_{\beta}}^{2}:=\sum_{n=0}^{\infty}2^{-nq}|\Phi^{(n)}|_{-p}^{2}<\infty,\quad p,q\in\mathbb{N}_{0}.
\]
The dual pairing between $\big(S_d'\big)_{\mu_{\beta}}^{-1}$ and $\big(S_d\big)_{\mu_{\beta}}^{1}$,
denoted by $\langle\!\langle\cdot,\cdot\rangle\!\rangle_{\mu_{\beta}}$
is a bilinear extension of scalar product in $L^{2}(\mu_{\beta})$.
For any $\varphi\in\big(S_d\big)_{\mu_{\beta}}^{1}$ and $\Phi\in\big(S_d'\big)_{\mu_{\beta}}^{-1}$
we have 
\[
\langle\!\langle\Phi,\varphi\rangle\!\rangle_{\mu_{\beta}}=\sum_{n=0}^{\infty}n!\langle\Phi^{(n)},\varphi^{(n)}\rangle.
\]
The set of $\mu_{\beta}$-exponentials 
\[
\left\{ e_{\mu_{\beta}}(\cdot,\varphi):=\frac{e^{\langle\cdot,\varphi\rangle}}{\mathbb{E}\big(e^{\langle\cdot,\varphi\rangle}\big)},\;\varphi\in S_{d,\mathbb{C}},\;|\varphi|_{p}<2^{-q}\right\} 
\]
forms a total set in $(H_{p})_{q,\mu_{\beta}}^{1}$ and for any $\varphi\in S_{d,\mathbb{C}}$
such that $|\varphi|_{p}<2^{-q}$ we have $\|e_{\mu_{\beta}}(\cdot,\varphi)\|_{p,q,\mu_{\beta}}<\infty$.

Let us introduce an integral transform, the $S_{\mu_{\beta}}$-transform,
which is used to characterize the spaces $(S_d)_{\mu_{\beta}}^{1}$
and $(S_d)_{\mu_{\beta}}^{-1}$. For any $\Phi\in(S_d)_{\mu_{\beta}}^{-1}$
and $\varphi\in U\subset S_{d,\mathbb{C}}$, where $U$
is a suitable neighborhood of zero, we define 
\[
S_{\mu_{\beta}}\Phi(\varphi):=\frac{\langle\!\langle\Phi,e^{\langle\cdot,\varphi\rangle}\rangle\!\rangle_{\mu_{\beta}}}{\mathbb{E}\big(e^{\langle\cdot,\varphi\rangle}\big)}=\frac{1}{E_{\beta}(\frac{1}{2}\langle\varphi,\varphi\rangle)}\langle\!\langle\Phi,e^{\langle\cdot,\varphi\rangle}\rangle\!\rangle_{\mu_{\beta}}.
\]
The characterization theorem for the space $(S_d)_{\mu_{\beta}}^{-1}$
via the $S_{\mu_{\beta}}$-transform is done using the spaces of holomorphic
functions on $S_{d,\mathbb{C}}$. We denote by $\mathrm{Hol}_{0}\big(S_{d,\mathbb{C}}\big)$
the space of holomorphic functions at zero where we identify two functions
which coincides in a neighborhood of zero. The space $\mathrm{Hol}_{0}\big(S_{d,\mathbb{C}}\big)$
is given as the inductive limit of a family of normed spaces, see
\cite{KSWY95} for the details and the proof of the following characterization
theorem. 
\begin{thm}[{cf.\ \cite[Theorem~8.34]{KSWY95}}]
 \label{theorem:Charact_distributions}The $S_{\mu_{\beta}}$-transform
is a topological isomorphism from $(S_d)_{\mu_{\beta}}^{-1}$ to $\mathrm{Hol}_{0}\big(S_{d,\mathbb{C}}\big)$.
\end{thm}

As a corollary from the characterization theorem the following integration result can be deduced.
\begin{thm}\label{it}
Let $(T,\mathcal{B},\nu)$ be a measure space and $\Phi_t\in (S_d)_{\mu_{\beta}}^{-1}$ for all $t\in T$. Let $\mathcal{U}\subset S_{d,\mathbb{C}}$ be an appropriate neighbourhood of zero and $0<C<\infty$, such that
\begin{enumerate}
\item $S_{\mu_{\beta}}\Phi_{\cdot}(\xi):T\to \mathbb{C}$ is measurable for all $\xi \in \mathcal{U}$.
\item $\int_T\left| S_{\mu_{\beta}}\Phi_t(\xi)\right|\, d\nu(t)\le C$ for all $\xi\in \mathcal{U}$.
\end{enumerate}
Then, there exists $\Phi\in (S_d)_{\mu_{\beta}}^{-1}$ such that for all $\xi\in \mathcal{U}$
\[ S_{\mu_{\beta}}\Psi(\xi)=\int_TS_{\mu_{\beta}}\Phi_t(\xi)\, d\nu(t). \]
We denote $\Psi$ by $\int_T\Phi_t\, d\nu(t)$ and call it the weak integral of $\Phi$.
\end{thm}

In the following we will use the $T_{\mu_{\beta}}$-transform which is defined as follows.
\begin{lem}
Let $\Phi \in (S_d)_{\mu_{\beta}}^{-1}$ and $p,q\in \mathbb{N}$ such that $\Phi \in  (H_{-p})_{-q,\mu_{\beta}}^{-1}$. Then, the $T_{\mu_{\beta}}$-transform given by
\[ T_{\mu_{\beta}}\Phi(\varphi)=\langle\!\langle \Phi, \exp\left( i\langle \cdot,\varphi\rangle\right) \rangle\!\rangle_{\mu_{\beta}} \]
is well-defined for $\varphi\in U_{p,q}$ and we have
\[  T_{\mu_{\beta}}\Phi(\varphi)=E_{\beta}\left(-\frac{1}{2}\langle \varphi,\varphi\rangle \right)S_{\mu_{\beta}}\Phi(i\varphi).\]
In particular, $T_{\mu_{\beta}}\Phi\in \mathrm{Hol}_{0}\big(S_{d,\mathbb{C}}\big)$ if and only if $S_{\mu_{\beta}}\in \mathrm{Hol}_{0}\big(S_{d,\mathbb{C}}\big)$. Moreover, Theorem \ref{it} also holds if the $S_{\mu_{\beta}}$-transform is replaced by the $T_{\mu_{\beta}}$-transform.
\end{lem}
\noindent For details and proofs we refer to \cite{GJRS14}.

\section{Self-intersection local times for ggBm in dimension $d$}

In this section we consider the self-intersection local times for ggBm which is formally given by
\[
L^{\beta,\alpha}(t):=\int_{0}^{t}\int_{0}^{t}\delta(B^{\beta,\alpha}(s)-B^{\beta,\alpha}(u))\,du\,ds.
\]
The (generalized) random variable $L_{\alpha,\beta}(a,t)$ is intended to measure the amount of time in which the sample path of a ggBm spends intersecting itself within the time interval $[0,t]$. A priori the expression above has no mathematical meaning since Lebesgue integration of Dirac delta distribution is not defined. We will prove that we can make sense of this object as a weak integral in Kondratiev distribution space. 
\begin{thm}
Let $0<\alpha<1$, $0<\beta \le 1$, and $d\in \mathbb{N}$ be such that $d\alpha <2$. Then, $L^{\beta,\alpha}(t)$ is a well defined element in $(S_{d})_{\mu_{\beta}}^{-1}$ in the weak sense. 
\end{thm}
\begin{proof}
Using the representation
\[
\delta(B^{\beta,\alpha}(s)-B^{\beta,\alpha}(u))=\left(\frac{1}{2\pi}\right)^{d/2}\int_{\mathbb{R}^{d}}e^{i(\lambda,B^{\beta,\alpha}(s)-B^{\beta,\alpha}(u))}\,d\lambda
\]
and denoting $\eta_{x}:=M_{-}^{\nicefrac{\alpha}{2}}1\!\!1_{[0,x)}$ we compute for any $\varphi \in \mathcal{U}$ the $T_{\mu_{\beta}}$-transform of $L^{\beta,\alpha}(t)$ to obtain
with 
\begin{eqnarray*}
(T_{\mu_{\beta}}L^{\beta,\alpha}(t))(\varphi) & = & \left(\frac{1}{2\pi}\right)^{d/2}\int_{0}^{t}\int_{0}^{t}T_{\mu_{\beta}}\delta(B^{\beta,\alpha}(s)-B^{\beta,\alpha}(u))(\varphi)\,du\,ds\\
 & = & \left(\frac{1}{2\pi}\right)^{d/2}\int_{0}^{t}\int_{0}^{t}\prod_{i=1}^{d}\bigg[\int_{\mathbb{R}}E_{\beta}\bigg(-\frac{1}{2}\lambda_{i}^{2}|\eta_{s}-\eta_{u}|^{2}-\frac{1}{2}\langle\varphi_{i},\varphi_{i}\rangle\\
 &  & -\lambda_{i}\langle\varphi_{i}\eta_{t}-\eta_{u}\rangle\bigg)d\lambda_{i}\bigg]du\,ds.
\end{eqnarray*}
By using the Laplace transform of $E_{\beta}$ (\ref{eq:LaplaceT_MWf}) and computing the Gaussian integral yields 
\begin{eqnarray*}
 &  & \int_{\mathbb{R}}E_{\beta}\bigg(-\frac{1}{2}\lambda_{i}^{2}|\eta_{s}-\eta_{u}|^{2}-\frac{1}{2}\langle\varphi_{i},\varphi_{i}\rangle-\lambda_{i}\langle\varphi_{i},\eta_{t}-\eta_{u}\rangle\bigg)d\lambda_{i}\\
 & = & \int_{\mathbb{R}}\int_{0}^{\infty}M_{\beta}(\tau)\exp\left(-\tau\left(\frac{1}{2}\lambda_{i}^{2}|\eta_{s}-\eta_{u}|^{2}+\frac{1}{2}\langle\varphi_{i},\varphi_{i}\rangle+\lambda_{i}\langle\varphi_{i},\eta_{t}-\eta_{u}\rangle\right)\right)d\tau d\lambda_{i}\\
 & = & \int_{0}^{\infty}d\tau M_{\beta}(\tau)\int_{\mathbb{R}}\exp\left(-\tau\left(\frac{1}{2}\lambda_{i}^{2}|\eta_{s}-\eta_{u}|^{2}+\frac{1}{2}\langle\varphi_{i},\varphi_{i}\rangle+\lambda_{i}\langle\varphi_{i},\eta_{t}-\eta_{u}\rangle\right)\right)d\lambda_{i}\\
 & = & \exp\left(\frac{1}{2}\langle\varphi_{i},\varphi_{i}\rangle\right)\int_{0}^{\infty}M_{\beta}(\tau)\sqrt{\frac{2\pi}{\tau|\eta_{s}-\eta_{u}|^{2}}}\exp\left(\frac{\tau^{2}\langle\varphi_{i},\eta_{t}-\eta_{u}\rangle^{2}}{2\tau|\eta_{s}-\eta_{u}|^{2}}\right)d\tau.
\end{eqnarray*}
Next, by using Lemma A.4 from \cite{GJRS14} and applying the Cauchy-Schwarz inequality we obtain
\begin{eqnarray*}
 &  & \int_{\mathbb{R}}E_{\beta}\bigg(-\frac{1}{2}\lambda_{i}^{2}|\eta_{s}-\eta_{u}|^{2}-\frac{1}{2}\langle\varphi_{i},\varphi_{i}\rangle-\lambda_{i}\langle\varphi_{i},\eta_{t}-\eta_{u}\rangle\bigg)d\lambda_{i}\\
 & = & \left(\frac{2\pi}{|\eta_{s}-\eta_{u}|^{2}}\right)\exp\left(\frac{1}{2}\langle\varphi_{i},\varphi_{i}\rangle\right)\int_{0}^{\infty}\tau^{-\nicefrac{1}{2}}M_{\beta}(\tau)\exp\left(\frac{\tau\langle\varphi_{i},\eta_{t}-\eta_{u}\rangle^{2}}{2|\eta_{s}-\eta_{u}|^{2}}\right)d\tau\\
 & \le & \left(\frac{2\pi}{|\eta_{s}-\eta_{u}|^{2}}\right)\exp\left(\frac{1}{2}\langle\varphi_{i},\varphi_{i}\rangle\right)\int_{0}^{\infty}\tau^{-\nicefrac{1}{2}}M_{\beta}(\tau)\exp\left(\frac{\tau|\varphi_{i}|^{2}}{2}\right)d\tau\\
 & \leq & K\left(\frac{2\pi}{|\eta_{s}-\eta_{u}|^{2}}\right)\exp\left(\frac{1}{2}\langle\varphi_{i},\varphi_{i}\rangle\right).
\end{eqnarray*}
Putting all together gives
\begin{eqnarray*}
|(T_{\mu_{\beta}}L^{\beta,\alpha}(t))(\varphi)| & \leq & \left(\frac{1}{2\pi}\right)^{d/2}\int_{0}^{t}\int_{0}^{t}\prod_{i=1}^{d}K\frac{\sqrt{2\pi}}{|\eta_{s}-\eta_{u}|}\exp\left(\frac{1}{2}\langle\varphi_{i},\varphi_{i}\rangle\right)duds\\
 & = & K\exp\left(\frac{1}{2}|\varphi|^{2}\right)\int_{0}^{t}\int_{0}^{t}\left(\frac{1}{|\eta_{s}-\eta_{u}|^{2}}\right)^{\nicefrac{d}{2}}duds\\
 & = & 2K\exp\left(\frac{1}{2}|\varphi|^{2}\right)\int_{0}^{t}\int_{0}^{s}(s-u)^{-\nicefrac{\alpha d}{2}}duds\\
\end{eqnarray*}
The last integral in finite for $\alpha d<2$. The announced result now follows by an application of Theorem \ref{it}.
\end{proof}
\section{Conclusion}
In this paper we have studied self-intersection local times of ggBm for the case $d\alpha <2$ and characterized it as a Mittag-Leffler distribution in a suitable distribution space. The case $ d\alpha <2$ corresponds for the Gaussian case $\beta=1$ to the case $Hd<1$. Indeed in this case \cite{HN05} showed that the self-intersection local time is a square integrable function. For $d\alpha \geq 1$ further renormalizations are needed, like e.g.~centering of the random variable. These considerations for the case of ggBm are postponed for a forthcoming paper.

\bibliographystyle{alpha}

\end{document}